\documentclass{article}
\usepackage{dirtytalk}
\usepackage{pdfpages}
\usepackage{authblk}
\usepackage[T1]{fontenc}
\usepackage[latin9]{inputenc}
\usepackage{amsfonts}
\usepackage{amsmath}
\usepackage{amssymb}
\usepackage{amsthm}
\usepackage{fullpage}
\usepackage{tikz}
\usepackage{graphicx}
\usepackage{datetime}
\usepackage[colorlinks=true,linkcolor=blue,citecolor=red]{hyperref}
\usepackage{bbm}
\usepackage{mathtools,bm,etoolbox}

\usepackage{thmtools}
\usepackage{thm-restate}

\usepackage{hyperref}

\usepackage[capitalize,nameinlink]{cleveref}

\newcommand\inote[1]{\textcolor[rgb]{0,.8,0}{({ Irit: }{#1}~)}}
\newcommand\nnote[1]{\textcolor[rgb]{0.8,.3,0}{({Irit v2: }{#1}~)}}

\newcommand{\eps}{\varepsilon}

\newcommand{\norm}[1]{\left\|{#1}\right\|}
\newcommand{\abs}[1]{|{#1}|}

\newtheorem{theorem}{Theorem}
\usepackage{amsthm}
\newtheorem*{claim*}{Claim}
\newtheorem{claim}[theorem]{Claim}
\newtheorem{lemma}[theorem]{Lemma}
\newtheorem{corollary}[theorem]{Corollary}

\newtheorem{question}[theorem]{Question}

\newenvironment{remark}{\noindent{\bf Remark.}\hspace*{1em}}{\bigskip}

\newenvironment{proof-sketch}{\noindent{\bf Sketch of Proof}\hspace*{1em}}{\qed\bigskip}
\newenvironment{proof-idea}{\noindent{\bf Proof Idea}\hspace*{1em}}{\qed\bigskip}
\newenvironment{proof-attempt}{\noindent{\bf Proof Attempt}\hspace*{1em}}{\qed\bigskip}

\newcommand\pare[1]{\left(#1\right)}

\providecommand\given{}

\DeclarePairedDelimiterXPP\prob[1]{
  {\textbf{Pr}}
}{[}{]}{}{
\renewcommand\given{  \nonscript\:
  \delimsize\vert
  \nonscript\:
  \mathopen{}
  \allowbreak}
  #1
}
\usepackage{thmtools,thm-restate}

\usepackage{xcolor,etoolbox}
\usepackage[version=4]{mhchem}
\usepackage{float}

\definecolor{shadecolor}{RGB}{254,255,154}
\newcommand{\rnote}[1]{\par\noindent\colorbox{shadecolor}
{\parbox{\dimexpr\textwidth-2\fboxsep\relax}{(Ron: #1)}}}

\iftrue
    \renewcommand{\inote}[1]{}
    \renewcommand{\rnote}[1]{}
    \renewcommand{\nnote}[1]{}
\fi

\usepackage[
backend=biber,
style=alphabetic,
citestyle=alphabetic,
maxbibnames=9
]{biblatex}
\newcommand{\routeprod}{\circ}

\title{Bipartite unique-neighbour expanders via Ramanujan graphs}
\date{}
\author{Ron Asherov}
\author{Irit Dinur\thanks{Irit Dinur acknowledges support by ERC grant 772839 and ISF grant 2073/21.}}
\affil{Weizmann Institute, Rehovot, Israel}
\begin{document}
\maketitle
\begin{abstract}
We construct an infinite family of bounded-degree bipartite unique-neighbour expander graphs with arbitrarily unbalanced sides.
Although weaker than the lossless expanders constructed by Capalbo et al., our construction is simpler and may be closer to be implementable in practice due to the smaller constants.
We construct these graphs by composing bipartite Ramanujan graphs with a fixed-size gadget in a way that generalizes
the construction of unique neighbour expanders
by Alon and Capalbo.
For the analysis of our construction we prove a strong upper bound on average degrees in small induced subgraphs of bipartite Ramanujan graphs. Our bound generalizes Kahale's average degree bound to bipartite Ramanujan graphs, and may be of independent interest. 
Surprisingly, our bound strongly relies on the  exact Ramanujan-ness of the graph and is not known to hold for nearly-Ramanujan graphs.
\end{abstract}

\section{Introduction}
An infinite family $G_n = (L_n \sqcup R_n, E_n)$ of $(c,d)$-biregular graphs
with $\abs{L_n} + \abs{R_n} \to \infty$
is called a \emph{unique neighbour expander family}
if there exists $\delta > 0$
such that for every $n$
and every set of left side vertices $S \subseteq L_n$
of size $\abs{S} \leq \delta \abs{L_n}$
there exists a unique neighbour of $S$ in $G_n$, namely a vertex in $R_n$ that is connected to exactly one vertex in $S$.
We only require that sets of left vertices have unique neighbours, and arbitrarily small right side sets may have no unique neighbour.

Alon and Capalbo \cite{alon2002explicit} construct several explicit families of unique neighbour expanders, via an elegant composition of a Ramanujan graph and a gadget.
They construct three\inote{check} families of general (non-bipartite) graphs in which all small sets have unique neighbours, and one family 
of slightly unbalanced bipartite graphs where small sets on the left have unique neighbors on the right. In their construction the left side is $22/21$ times bigger than the right side.
The more imbalanced the graph, the harder it is for small left hand side sets to expand into the right hand side. Capalbo et. al. \cite{capalbo2002randomness} construct arbitrarily unbalanced bipartite graphs that are lossless expanders, a notion strictly stronger than unique neighbour expansion. Their construction is based on a sequence of somewhat involved composition steps using randomness conductors.  

Our main theorem is an efficient construction of an infinite family of bipartite unique neighbour expanders for any constant imbalance $\alpha$, and any sufficiently large left-regularity degrees of a specific form:
\begin{theorem}\label{une}
There is a function $\hat{q}: \mathbb{N} \times \mathbb{R} \to \mathbb{N}$
such that for every integer $c_0>5$ and real number $\alpha > 1$,
if $q > \hat{q}(c_0,\alpha)$ is a prime power and
$\alpha c_0 (q+1)$ is an integer,
then there is a polynomial-time construction of an infinite family of $(c_0(q+1), \alpha c_0(q+1))$-biregular unique neighbour expanders.
\end{theorem}
The theorem is proven in \cref{thmproof},
and provides a way to compute $\hat q (c_0, \alpha)$.
Here are some computed values of $\hat q(c_0,\alpha)$ for several values of $c_0, \alpha$.

\begin{table}[H]
\centering
 \begin{tabular}{c c c} 
 $c_0$ & $\alpha$ & $\hat q(c_0, \alpha)$ \\ [0.5ex] 
 \hline
 10 & 2 & 18907 \\
 35 & 2 & 1492 \\
 100 & 100 & 136051 \\
 100 & 1.01 & 1135 \\ [1ex]
 \end{tabular}
\end{table}
\inote{maybe give some annotation - why did you choose to put these number here?}
\rnote{No particular reason, just numbers. I'm adding this short paragraph below about it}
Notice that $\hat q(c_o, \alpha)$ increases with $\alpha$, reflecting the fact that constructions with larger $\alpha$ (namely, more imbalanced sides) are harder to come by, and require larger degrees.

The construction uses an infinite family of bipartite Ramanujan graphs,
namely graphs whose non-trivial spectrum is contained in the spectrum of the $(c,d)$-biregular tree
(see \nameref{prelim} for details).
We construct the unique neighbour expander family by taking a family of bipartite Ramanujan graphs and combining them with a fixed size graph (\say{gadget}),
with a good unique neighbour property (small sets have unique neighbours),
whose existence is shown via the probabilistic method (\cref{gadget}).
The combination is done as follows.
We first place a copy of the gadget for every right side vertex of the Ramanujan graph.
The vertex is replaced by the right side of the gadget,
and its neighbours are identified with the left side of the gadget.
The gadget is used to route the neighbours of each left side vertex in the Ramanujan graph to its neighbours in the product graph. 

Expansion in the product graph comes from unique neighbour expansion of the gadget together with low degree vertices in the Ramanujan graph. Sufficiently low degree vertices are guaranteed to exist thanks to the following (new) bound on the average degree of induced subgraphs of bipartite Ramanujan graphs, which may be of independent interest.
\begin{restatable}{theorem}{maintech}\label{thm:main-tech}
Let $G = (L \sqcup R, E)$ be a $(c,d)$-biregular Ramanujan graph, and let $\eps > 0$.
Then there exists $\delta>0$,
that depends only on $\eps, c, d$,
such that for every $S \subset L$ of size $\abs{S} \leq \delta \abs{L}$,
the set $N(S) \subseteq R$ of the neighbours of $S$ satisfies
\begin{equation*}
    \frac{c\abs{S}}{\abs{N(S)}} \leq 1 + (1+\eps)\sqrt{\frac{d-1}{c-1}}
    .
\end{equation*}
\end{restatable}
The theorem shows that every small set on the left side admits neighbours on the right side with low degree
in the induced subgraph.
The proof involves recursive analysis of non-backtracking paths. Interestingly, the recursion has a nice solution only when the graph is Ramanujan. It is unclear whether this method can be extended to \say{nearly-Ramanujan} graphs.

Combining the average degree upper bound with the gadget,
the low-degree right-side vertices in the Ramanujan graph imply a small set of left-side vertices in the gadget;
this set will have a unique neighbour in the gadget,
which gives (via \cref{unq_claim}) a unique neighbour in the constructed graph.

Even though Ramanujan graphs are the best spectral expanders one can hope for,
an efficient construction of Ramanujan graphs (be them bipartite or not)
does not immediately imply
that we can construct unique neighbour expanders.
In the $d$-regular case,
Kahale shows (\cite[Thm 5.2]{kahale1995eigenvalues}) that there are nearly-Ramanujan graphs with expansion at most $d/2$,
which is not enough for unique neighbour expansion.
In fact, recently Kamber and Kaufman \cite{kamber2022combinatorics} proved that some Ramanujan graphs strongly fail to have unique neighbour expansion, by giving explicit constructions of arbitrarily small sets that do not admit a unique neighbour.

As mentioned, the graph product we define requires a fixed size gadget,
whose proof of existence is not constructive.
In principle, such a gadget could be found by exhaustive search since we are working in a constant size search space. The gadget's size in our construction is at least cubic in $q$,
so exhaustive search is impractical for even small values of $q$. Unfortunately we know of no efficient construction of a gadget
with the required parameters.
It is possible that the graph sampling method present in \cite{applebaum2019sampling} can be used to construct fixed size gadgets more efficiently.

The rest of this work is organized as follows.
In \cref{related_work} we survey some of the uses of unique neighbour expanders,
and mention known constructions of such graphs.
\cref{prelim} provides basic definitions and results.
Our main technical tool,
that asserts the low induced degree in bipartite Ramanujan graph,
is stated and proven in \cref{expansion}.
We prove the existence of a fixed-size gadget with good unique neighbour expansion
properties in \cref{gadget_section}.
In \cref{construction} we define the way we use the Ramanujan graphs
and the gadget to construct bipartite unique neighbour expanders,
and by that prove \cref{une}.

\section{Related work}\label{related_work}
One of the prominent uses of bipartite expanders in general and bipartite unique neighbour expanders in particular,
and the motivation for this work,
is the construction of error correcting codes.
The works of Tanner \cite{tanner1981recursive} and later Sipser and Spielman \cite{sipser1996expander}
construct linear error correcting codes $\mathcal{C}(B,C_0)$ from a bipartite
graph $B$ and a smaller linear code $C_0$. 
It is shown that under some assumptions on the code $C_0$
and the expansion properties of the bipartite graph $B$,
the resulting code has good distance.
This gives a way to take a family of graphs and transform it into a family of codes.
Our work describes a construction that, in a sense, goes the other way around:
given two bipartite graphs, $B$ and $B_0$, we view $B_0$ as a parity check graph\footnote{This is a bipartite graph whose incidence structure is given by the parity check matrix.} of the base code $C_0$, and $B$ plays the role of the underlying graph of a Tanner code  $\mathcal{C}(B,C_0)$. Our output graph is just the parity check graph of  $\mathcal{C}(B,C_0)$.
We give full details of this graph product in \cref{rpd}.

In \cite{dinur2006robust,ben2009tensor} it is shown that codes constructed  on top of unique neighbour expanders
are weakly smooth and can be used to construct robustly testable codes.
But the uses of unique neighbour expanders are not limited to error correcting codes:
for example, such graphs may be used in the context of non-blocking networks,
where it is required to connect several input-output terminals via paths in a non-intersecting fashion.
Arora et al. \cite{arora1996line} use graphs with expansion beyond the $d/2$ barrier
to establish the existence of unique neighbours in the graph,
which are useful in finding input-output paths in the online settings.
Roughly speaking, when routing a set of input-output pairs,
the algorithm can use all unique neighbours freely since they
are guaranteed not to interfere with any other paths.
Pippenger \cite{pippenger1993self} uses explicit constructions of spectral expanders
in order to solve a similar problem,
in the case where the route planning is computed locally.
There the spectral expansion of a graph is proven to imply a combinatorial expansion,
in a similar way to our \cref{thm:main-tech}.

Another use for unique neigbhour expanders is for load-balancing problems,
such as the token distribution problem described in \cite{peleg1989token},
and the similar pebble distribution problem,
briefly discussed in \cite{alon2002explicit}.
In the latter, pebbles are placed arbitrarily on vertices of a graph,
and need to be distributed via edges of the graph such that no vertex has more than one pebble.
Given that the total number of pebbles is small and that the graph has the unique neighbour property,
we have an efficient parallel algorithm for redistributing the pebbles.

Alon and Capalbo \cite{alon2002explicit} construct several families of unique neighbour expanders,
one of them is a family of bipartite graphs whose left side is $22/21$ times bigger than the right side.
Similar to the construction presented at this work,
each graph in the constructed family is a combination of a Ramanujan graph and a fixed graph.
These graphs are not (bi-)regular but their degrees are bounded by a constant.
Becker \cite{becker} uses a different family of $8$-regular Ramanujan graphs
in order to construct a family of (non-bipartite) unique neighbour expanders,
with the additional property that each graph in the family
is a Cayley graph.

A different approach to constructing bipartite graphs uses randomness conductors.
Randomness conductors are functions that receive a bitstring with some entropy (according to some measure of entropy),
and a uniformly random bitstring,
and output a bitstring, with certain guarantees on its entropy.
Some conductors can be constructed explicitly via a spectral method,
and Capalbo et al. \cite{capalbo2002randomness} combine them in a zig-zag-like fashion
in order to construct an infinite family of \emph{bipartite lossless expanders},
namely bipartite graphs with fixed left-regularity $c$ where small enough sets 
contained in the left side
have at least $c(1-\eps)$ neighbours on the right side.
These graphs are trivially unique neighbour expanders,
since a simple counting argument shows that if a set expands by a factor of more than $c/2$,
then it has unique neighbours.

\section{Preliminaries}\label{prelim}
\subsection{Expander graphs}
In this work we deal with undirected graphs,
that may contain multiple edges between two vertices,
but do not contain self-loops.
For a graph $G$ and a subset of its vertices $S$ we denote by $N_G(S)$ the \emph{neighbourhood} of $S$,
namely all vertices adjacent to some vertex in $S$.
When the graph in discussion is obvious, we may omit it and write $N(S)$.
We say that $v$ is a \emph{unique neighbour} of $S$ if there is a unique $u \in S$ that is adjacent to $v$.

Let $(G_n)$ be a series of graphs with the number of vertices growing to infinity.
There are several well studied notions of \emph{expansion} in graph families;
we note some of them.
\nnote{in prelims of a paper (unlike a thesis) we put only what we actually use. so we should probably only discuss bipartite graph, and only spectral, lossless and unique nbr expansion}
\begin{enumerate}
    \item \emph{Vertex expansion.}
    $(G_n)$ is a $(\delta,\alpha)$-vertex expander if for every $n$ and any subset $S \subseteq V_{G_n}$,
    if $\abs{S} \leq \delta \abs{V_{G_N}}$ we have that $\abs{N_{G_N}(S)} \geq \alpha \abs{S}$.
    \item \emph{Edge expansion.}
    $(G_n)$ is a $(\delta,\alpha)$-edge expander if for every $n$ and any subset $S \subseteq V_{G_n}$,
    if $\abs{S} \leq \delta \abs{V_{G_N}}$
    we have that at least an $\alpha$-fraction of the edges with one endpoint in $S$
    have their other endpoint outside of $S$.
    \item \emph{Spectral expansion.}
    Assume that $(G_n)$ are all $d$-regular,
    and let $A_n$ be the adjacency operator associated with $G_n$,
    so $A_n$ is indexed by vertices of $G_n$ and ${(A_n)}_{uv}$ counts how many edges there are
    between $u$ and $v$ in $G_n$.
    Let $\lambda_1 \geq \ldots \geq \lambda_{V_n}$ be its spectrum.
    It can be seen that $\lambda_1 = d$.
    Then $(G_n)$ is a $\lambda$-spectral expander if for all $n$ and $i \neq 1$ we have $\abs{\lambda_i} \leq \lambda$.
    \item \emph{Unique neighbour expansion.}
    $(G_n)$ is a $\delta$-unique neighbour expander if for every $n$, any subset $S \subseteq V_{G_n}$
    of size at most $\delta \abs{V_{G_N}}$ has a unique neighbour.
\end{enumerate}
These definitions apply to bipartite graphs $G_n = (L_n \sqcup R_n, E_n)$ as well,
with the exception that we usually consider sets contained in the left side only,
and require that $L_n / R_n$ is a constant, normally greater than $1$.
In this case we note that edge expansion is meaningless (since all edges leaving the left side enter the right side),
and if a bipartite graph is $(c,d)$-biregular,
namely if all left-side vertices have degree $c$
and all right-side vertices have degree $d$,
then the largest eigenvalue of the associated adjacency operator is $\sqrt{cd}$.

It can be seen that for $d$-regular graphs, the best spectral expansion we can hope for is $\alpha=2\sqrt{d-1}$.
These graphs are known as \emph{Ramanujan graphs}.
\nnote{Add a subsection discussing known constructions of bipartite Ramanujan graphs, and a quote of the theorem of Balantine that we are using.
Also, explain the subtleties of what it means to be bip ramanujan, and that the MSS work doesn't suffice}
\rnote{done}
\subsection{Bipartite Ramanujan graphs}
Ramanujan graphs have the best spectral gap \cite{nilli1991second},
and their non-trivial eigenvalues are contained in the spectrum of the infinite $d$-regular tree $T_d$.
Similarly, in the bipartite case,
Biregular Ramanujan graphs are defined via their relation to the infinite biregular trees:
the infinite $(c,d)$-biregular tree $T_{c,d}$, for $d > c$, has the spectrum
\begin{equation*}
    \lambda \in \operatorname{spec}(T_{c,d}) \Leftrightarrow
    \abs{\lambda} \in \{0\} \cup \left[\sqrt{d-1} - \sqrt{c-1}, \sqrt{d-1}+\sqrt{c-1}\right]
\end{equation*}
(see, e.g., \cite{godsil1988walk}, \cite{li1996}.)
We therefore say that a finite $(c,d)$-biregular graph is \emph{bipartite Ramanujan} if its nontrivial eigenvalues lie in this set.
That means that every eigenvalue $\lambda$ of a bipartite Ramanujan graph belongs to one of these classes:
\begin{enumerate}
    \item Trivial: $\lambda = \pm \sqrt{cd}$, with eigenvectors fixed on either sides, or $\lambda=0$;
    \item $\lambda \in [\sqrt{d-1}-\sqrt{c-1}, \sqrt{d-1}+\sqrt{c-1}]$ are the nontrivial positive eigenvalues;
    \item $\lambda \in [-\sqrt{c-1}-\sqrt{d-1}, \sqrt{c-1}-\sqrt{d-1}]$ are the nontrivial negative eigenvalues.
    Note that since the graph is bipartite, $\lambda$ is an eigenvalue if and only if $-\lambda$ is an eigenvalue.
\end{enumerate}
By an extension of the Alon-Boppana bound, given in \cite{feng1996spectra}, this is the best spectral gap we can hope for, at least as far as upper bounds for $\abs{\lambda}$ are concerned.
We note that unlike the $d$-regular case,
we require a lower bound to $\abs{\lambda}$ too,
which is essential for our proof.

While there is a vast literature on the construction of $d$-regular Ramanujan graph (most prominently \cite{lubotzky1988ramanujan} and \cite{margulis1988explicit}), less is known about bipartite Ramanujan graphs.
In 2014 Marcus et al. \cite{marcus2013interlacing} proved the existence of biregular graphs with one-sided spectral graphs that resemble the Ramanujan bounds:
these graphs satisfy the one-sided inequality only,
namely $\abs{\lambda} \leq \sqrt{d-1} + \sqrt{c-1}$ for every nontrivial eigenvalue $\lambda$.
Gribinski et al. \cite{gribinski2021existence} showed a polynomial-time construction of such graphs, for every degrees $(d, kd)$ for any integers $d, k$.
These graphs do not suffice for our analysis, since we make explicit use of the lower bound $\abs{\lambda} \geq \sqrt{d-1}-\sqrt{c-1}$ too.

In 2021 Brito et al. \cite{brito2022spectral} proved that a random biregular graph is nearly Ramanujan with high probability.
Interestingly, and unlike other works in this field,
our proof strongly relies on the graph to be exactly Ramanujan,
so we cannot use those constructions either.

We use an explicit construction of bipartite Ramanujan graphs (with both bounds on non-trivial eigenvalues) given by Ballantine et al.:
\begin{theorem}[\cite{ballantine2015explicit}]\label{ballentine}
For every prime power $q$,
there exists an explicit construction
of a $(q+1, q^3+1)$-biregular Ramanujan graph.
\end{theorem}

\section{Vertex expansion in biregular Ramanujan graphs}\label{expansion}
Our main technical tool is the following theorem showing that bipartite Ramanujan graphs exhibit excellent left-to-right expansion. We restate the theorem for convenience. 

\maintech*

We note that the quantity on the left hand side of the inequality can be interpreted as follows. Look at the bipartite graph induced by taking the vertices $S$ on the left and $N(S)$ on the right.
Since every left vertex has $c$ outgoing edges,
the total number of edges in the induced subgraph is $c \abs{S}$.
This means that the expression on the left hand side of the inequality is exactly the average degree of the right side of the induced subgraph.%
\inote{suggest to add:}\rnote{Thanks, I added a sentence at the end of this paragraph on the lower bounds thing too}
Interestingly, the bound in this theorem is strictly stronger than what we would get from just applying the expander mixing lemma which amounts to
\begin{equation*}
    \frac{c\abs{S}}{\abs{N(S)}} \leq (1+\eps)\cdot\left(1 + \frac {d-1}{c-1} + 2\sqrt{\frac{d-1}{c-1}}\right)
    .
\end{equation*} See \cref{claim:EML} for details.
The fact that we improve upon the expander mixing lemma is perhaps not surprising since our analysis is based on enumerating non-backtracking paths, and not just on magnitude of the second largest eigenvalue.
We also use lower bounds on the magnitude of all nontrivial eigenvalues,
whereas the expander mixing lemma uses just upper bounds.

\subsection{Comparison to known bounds}
\inote{I'd put comparison to EML first. It's the thing most readers will immediately try to compare to; and the favorable comparison makes a point that thm 2 in itself is a "secondary" result of this paper}
\rnote{I'm convinced, thanks!}

As noted above,
\cref{thm:main-tech} is an improvement of the bound that the expander mixing lemma gives in similar settings,
which only uses the one-sided inequality $\abs{\lambda} \leq \sqrt{d-1} + \sqrt{c-1}$.
For reference, we state and prove the expander mixing lemma for bipartite Ramanujan graphs.
\begin{claim}[Expander mixing lemma for bipartite Ramanujan graphs]\label{claim:EML}
Let $G = (L \sqcup R, E)$ be a $(c,d)$-biregular Ramanujan graph, and let $\eps > 0$.
Then there exists $\delta > 0$ such that for every $S \subseteq L$ of size $\abs{S} \leq \delta \abs{L}$,
the neighbourhood of $S$ satisfies
\begin{equation*}
    \frac{c \abs{S}}{\abs{N(S)}} \leq 
    (1+\eps) \pare{1 + \frac{d-1}{c-1} + 2\frac{\sqrt{d-1}}{\sqrt{c-1}}}
    .
\end{equation*}
\end{claim}
\begin{proof}
The expander mixing lemma for biregular graphs says that for every $S \subseteq L$, $T \subseteq R$ we have
\begin{equation*}
    \left|\frac{\abs{e(S,T)}}{\abs{E}} - \frac{\abs{S}}{\abs{L}} \cdot \frac{\abs{T}}{\abs{R}}\right| \leq
    \frac{\lambda}{\sqrt{cd}} \sqrt{\frac{\abs{S}}{\abs{L}} \cdot \frac{\abs{T}}{\abs{R}}}
\end{equation*}
where $\lambda$ is the second largest eigenvalue of $G$ (see, e.g., \cite{haemers1995interlacing}).
It is clarified that we consider the spectrum of $G$ as an adjacency operator\inote{what do you mean "as an adjacency operator"? maybe point to some explanation in the preliminaries}\rnote{Wrote about it in the prelim}, so the largest eigenvalue is $\sqrt{cd}$.

Picking $T = N(S)$ means all edges coming out from $S$ are in the cut, namely $\abs{e(S,T)} = c \abs{S}$.
Plugging that in gives
\begin{equation*}
    \left|\frac{c \abs{S}}{c \abs{L}} - \frac{\abs{S}}{\abs{L}} \cdot \frac{\abs{N(S)}}{\abs{R}}\right| \leq
    \frac{\lambda}{\sqrt{cd}} \sqrt{\frac{\abs{S}}{\abs{L}} \cdot \frac{\abs{N(S)}}{\abs{R}}}
    .
\end{equation*}
Multiplying both sides by $\frac{\abs{L}}{\abs{S}}$ gives
\begin{equation}\label{eml:ineq}
    \left| 1 - \frac{\abs{N(S)}}{\abs{R}} \right| \leq
    \frac{\lambda}{\sqrt{cd}} \sqrt{\frac{\abs{S}}{\abs{L}} \cdot \frac{\abs{N(S)}}{\abs{R}}}
    \cdot \frac{\abs{L}}{\abs{S}}
    =
    \frac{\lambda}{\sqrt{cd}} \sqrt{\frac{\abs{N(S)}}{\abs{R}} \cdot \frac{\abs{L}}{\abs{S}}} =
    \frac{\lambda}{\sqrt{cd}} \sqrt{\frac{\abs{N(S)}}{\abs{S}}} \cdot \sqrt{\frac{d}{c}} =
    \frac{\lambda}{c} \sqrt{\frac{\abs{N(S)}}{\abs{S}}}
\end{equation}
where we also used the fact that $\abs{E} = c \abs{L} = d \abs{R}$.

Let us assume that $\abs{S} = \alpha \abs{L}$. Then we can upper bound $\abs{N(S)}$ by
\begin{equation*}
    \abs{N(S)} \leq c \abs{S} =
    \alpha c \abs{L} =
    \alpha d \abs{R}
\end{equation*}
and so we have
\begin{equation*}
    1 - \frac{\abs{N(S)}}{\abs{R}} \geq
    1 - \frac{d \alpha \abs{R}}{\abs{R}} = 
    1 - d \alpha
    .
\end{equation*}
We square \eqref{eml:ineq} and plug in the last inequality to get
\begin{equation*}
    \pare{1 - d \alpha}^2 \leq \frac{\lambda^2}{c} \cdot \frac{\abs{N(S)}}{c\abs{S}}
    .
\end{equation*}
Recall that $G$ is bipartite Ramanujan, so $\abs{\lambda} \leq \sqrt{d-1} + \sqrt{c-1}$.
Use that and rearrange:
\begin{align*}
\frac{c\abs{S}}{\abs{N(S)}} &\leq
\frac{\lambda^2}{c} (1-d \alpha)^{-2} \\&\leq
\frac{d-1 + c-1 + 2\sqrt{d-1}\sqrt{c-1}}{c} (1-d\alpha)^{-2} \\&\leq
\frac{d-1 + c-1 + 2\sqrt{d-1}\sqrt{c-1}}{c-1} (1-d\alpha)^{-2} \\&=
\pare{1 + \frac{d-1}{c-1} + 2\frac{\sqrt{d-1}}{\sqrt{c-1}}} (1-d\alpha)^{-2}
.
\end{align*}
The claim is proven by noting that there is some $\delta > 0$ such that $(1-d \alpha)^{-2} \leq 1+\eps$ for every $\alpha < \delta$,
namely whenever $\abs{S} \leq \delta \abs{L}$.
\end{proof}
\inote{metzuyan! just slow enough for me !!!}
\rnote{:)}

Kahale proved (\cite[Thm 4.2]{kahale1995eigenvalues}) that in $d$-regular Ramanujan graphs (not necessarily bipartite),
small induced subgraphs have average degree at most $1 + \sqrt{d-1}$.
Interestingly, this result can be deduced almost immediately from \cref{thm:main-tech}.
This is due to the following lemma, proven in \nameref{app}, 
which asserts that the \emph{edge-vertex incidence graph} (see \cite{sipser1996expander})
of a $d$-regular Ramanujan graph is a $(2,d)$-biregular Ramanujan graph:
\begin{restatable}{lemma}{vve}\label{vve}
Let $G$ be a $d$-regular Ramanujan graph,
and $G'$ its edge-vertex incidence graph.
Then $G'$ is a $(2,d)$-biregular Ramanujan graph.
\end{restatable}

We state and prove Kahale's bound, but we will not use it in our construction.
\begin{corollary}
Let $G = (V_G, E_G)$ be a $d$-regular Ramanujan graph, and let $\eps > 0$.
Then there exists $\delta > 0$ such that for every induced subgraph $S$
with at most $\delta \abs{V_G}$ vertices, the average
degree of $S$ is at most
\begin{equation*}
    \overline{d_S} :=
    \frac{2\abs{E_S}}{\abs{V_S}} \leq
    1 + (1+\eps) \sqrt{d-1}
    .
\end{equation*}
\end{corollary}
\begin{proof}
Let $G = (V_G, E_G)$ be a $d$-regular Ramanujan graph
and $\eps > 0$.
We define $G' = (L_{G'} \sqcup R_{G'}, E_{G'})$ as the edge-vertex incidence graph,
namely $L_{G'} = E_G$, $R_{G'} = V_G$,
and for every edge $e = \{u,v\}$ in $G$
we have the two edges $\{e, u\}$ and $\{e, v\}$ in $G'$.
Since the degree of every vertex in $G$ is $d$,
and since every edge has two endpoints,
we have that $G'$ is a $(2,d)$-biregular graph.
\cref{vve}
asserts that $G'$ is Ramanujan in the bipartite sense.
By \cref{thm:main-tech}, there exists $\delta > 0$ such that if $T \subseteq L_{G'}$ is of size at most $\delta \abs{L_{G'}}$, then
\begin{equation*}
    \frac{2\abs{T}}{\abs{N_{G'}(T)}} \leq 1 + (1+\eps) \sqrt{d-1} .
\end{equation*}

A subgraph $S = (V_S, E_S)$ of $G$
satisfies that $E_S$ is a subset of left-side vertices in $G'$,
$V_S$ is a subset of right-side vertices in $G'$,
and $V_S = N_{G'}(E_S)$
(because if an edge is in the subgraph then both of its endpoints are in the subgraph, and we assume that the subgraph does not contain an isolated vertex).
Therefore, if $E_S$ is sufficiently small,
namely if $\abs{E_S} \leq \delta \abs{L_{G'}} = \delta \abs{E_G}$,
then by \cref{thm:main-tech} the average degree of $N_{G'}(E_S) = V_S$ is
bounded by $1 + (1+\eps)\sqrt{d-1}$.

We add that if we wish to find a bound the number of vertices,
we note that $\abs{E_S} \leq \frac{d}{2} \abs{V_S}$.
So every induced subgraph with no more than $\frac{2}{d} \delta \abs{E_G} = \delta \abs{V_G}$
vertices will satisfy the required average degree bound.
\end{proof}

\subsection{Proof of Theorem \ref{thm:main-tech}}

\cref{thm:main-tech} is proven by enumerating non-backtracking paths.
A non-backtracking path of length $\ell$ is a sequence of edges $((s(e_i), t(e_i)))_{i=1}^\ell$ such that for every $i$, $t(e_i) = s(e_{i+1})$ and $s(e_i) \neq t(e_{i+1})$.

For a bipartite graph $G$ and a subset $S$ of left side vertices we define $M_\ell(S)$ to be the number of all non-backtracking paths whose all left-side vertices are in $S$,
and $M_\ell(S,G)$ to be the number of non-backtracking paths whose first and last left-side vertices are in $S$.
Clearly $M_\ell(S) \leq M_\ell(S,G)$, as paths of the latter type may leave $S \sqcup N(S)$ (before re-entering $S$ at the last step).
We use a lower bound on $M_\ell(S)$ due to \cite{kamber2019}:
\nnote{This seems unrelated here, perhaps belongs further down, inside the proof of thm 2}
\rnote{reordered things -- first state lower \& upper bound, then use them to prove the thm, then the long technical part of getting the upper bound}
\begin{lemma}\label{amitay}
For every undirected bipartite graph $G = (L_G \sqcup R_G, E_G)$ and integer $l$ it holds that
\begin{equation*}
    M_\ell(L_G) \geq
    \abs{E_G}
    \pare{\sqrt{(\bar{d}_L-1)(\bar{d}_R-1)}}^{\ell-1}
\end{equation*}
where $\bar{d}_L, \bar{d}_R$ are the average degrees of the left and right sides of $G$ respectively.
\end{lemma}
We state and prove an upper bound on $M_\ell(S, G)$:
\begin{lemma}\label{upper_bound}
Let $G$ be a $(c,d)$-biregular Ramanujan graph with $n$ vertices on the left side, and $S$ a subset of the left side.
Then for every integer $\ell$:
\begin{equation*}
    M_{2\ell}(S,G) \leq \abs{S} \pare{(2+\sqrt{d-1})\ell+2} (c-1)^{\ell/2} (d-1)^{\ell/2}
\end{equation*}
provided that $S$ is small enough:
\begin{equation}\label{m_lemma:cond}
    \abs{S}(c-1)^{\ell/2}(d-1)^{\ell/2} \leq n
    .
\end{equation}
\end{lemma}
\nnote{better define $M(S)$ properly according to your needs and avoid "abuse of notation"}
\rnote{deleted remark}
Before proving the upper bound, we show how these bounds can be combined to obtain \cref{thm:main-tech}.
\begin{proof}[Proof of \cref{thm:main-tech}]
\nnote{nicer to use $\ell$ instead of $l$}
\rnote{done}
Let $\ell$ be an integer to be determined later, $S \subseteq L$ a sufficiently small subset
(where sufficiently smalls means \eqref{m_lemma:cond}).
Denote by $N(S) \subseteq R$ the neighbours of $S$.
The subgraph induced on $S \cup N(S)$ has $c\abs{S}$ edges,
with left degrees all $c$ and average right degree $\bar{d}_R = \frac{c \abs{S}}{\abs{N(S)}}$.
\nnote{need to explain where does each ineq come from
Combining theorem 6 and lemma 9 we have gives:}
\rnote{done}

Chaining the inequalities in \cref{amitay} and \cref{upper_bound}, we have
\begin{equation*}
c \abs{S} \left( (c-1) (\bar{d}_R-1) \right)^{\frac{2\ell-1}{2}} \leq
M_{2\ell}(S) \leq
M_{2\ell}(S,V_G) \leq
\abs{S} \cdot \pare{(2+\sqrt{d-1})\ell+2} \cdot (c-1)^{\ell/2} (d-1)^{\ell/2} 
.
\end{equation*}

Simplifying, we get,
\begin{align*}
c (c-1)^{\ell-\frac{1}{2}} (\bar{d}_R-1)^{\ell-\frac{1}{2}} &\leq
    \pare{(2+\sqrt{d-1})\ell+2} \cdot (c-1)^{\ell/2} \cdot (d-1)^{\ell/2} \\
    (\bar{d}_R-1)^{\ell-\frac{1}{2}} &\leq
    \frac{\pare{(2+\sqrt{d-1})\ell+2}\sqrt{c-1}}{c}\left(\sqrt{\frac{d-1}{c-1}}\right)^\ell \\
    \bar{d}_R - 1 &\leq
    {\pare{\underbrace{\frac{\pare{(2+\sqrt{d-1})\ell+2}\sqrt{c-1}\sqrt{\bar{d}-1}}{c}}_{\bigstar}}^{1/\ell}}
    \cdot \sqrt{\frac{d-1}{c-1}}
\end{align*}
Since $\bar{d} \leq d$, we have that $\bigstar = O(\ell)$,
so $\bigstar^{1/\ell} = O(1)$,
hence for a fixed $\eps > 0$ there exists a constant $\ell$
(that depends only on $\eps, c, d$)
such that $\bigstar^{1/\ell} \leq 1+\eps$;
this $\ell$ determines, via inequality \eqref{m_lemma:cond}, a fixed $\delta$
such that whenever $\abs{S} \leq \delta n$ we have
\begin{equation*}
    \bar{d}_R \leq 1 + (1+\eps) \sqrt{\frac{d-1}{c-1}}
    .
\end{equation*}
\end{proof}

We proceed to prove \cref{upper_bound}.

For a bipartite graph $G = (L_G \sqcup R_G, E_G)$ and an integer $\ell$, we define $A_\ell^{LL}, A_\ell^{LR}, A_\ell^{RL}, A_\ell^{RR}$ as operators corresponding to non-backtracking paths of length $\ell$, i.e.
\begin{equation*}
    A_\ell^{LL}: L^2(L_G) \to L^2(L_G) \qquad,\qquad
    (A_\ell^{LL}f)(x) = \sum\limits_{(e_1, \ldots, e_\ell), t(e_\ell) = x, s(e_1),t(e_\ell) \in L_G} f(s(e_1))
\end{equation*}
with the summation over all non-backtracking paths of length $\ell$, and similarly for the other operators.

Let $M$ be the operator corresponding to a single step from the right side $G$ to the left side of $G$,
namely $M$ has $\abs{R_G}$ rows and $\abs{L_G}$ columns,
with $M_{uv}$ counting the number of edges between $u \in R_G$ and $v \in L_G$ in $G$.
Then the following recursive formulae hold for every integer $\ell>1$:
\begin{align*}
    M^\top A_\ell^{LL} &= A_{\ell+1}^{RL} + (d-1) A_{\ell-1}^{RL}\\
    M^\top A_\ell^{LR} &= A_{\ell+1}^{RR} + (d-1) A_{\ell-1}^{RR}\\
    M A_\ell^{RL} &= A_{\ell+1}^{LL} + (c-1) A_{\ell-1}^{LL}\\
    M A_\ell^{RR} &= A_{\ell+1}^{LR} + (c-1) A_{\ell-1}^{LR}
\end{align*}
The first formula is explained as follows.
Every non-backtracking path from $R$ to $L$ of length $\ell+1$ is composed of
a non-backtracking path from $L$ to $L$ of length $\ell$ plus an extra step
(that's the $M^\top A_\ell^{LL}$ factor.)
The opposite is true, except for paths counted in $M^\top A_\ell^{LL}$
that do backtrack,
namely those made of a non-backtracking path of length $\ell-1$,
and walking back and forth along the same edge.
There are $d-1$ ways to choose that edge
(since it cannot be the one that was last in the path of length $\ell-1$,
otherwise it wouldn't be counted in $M^\top A_\ell^{LL}$),
so we need to subtract $(d-1)A_{\ell-1}^{RL}$.
The rest of the equations are explained in an analog way.

Due to symmetry we have:
\begin{equation*}
    (A_\ell^{LL})^\top = A_\ell^{LL} \qquad,\qquad
    (A_\ell^{RR})^\top = A_\ell^{RR} \qquad,\qquad
    (A_\ell^{LR})^\top = A_\ell^{RL}
\end{equation*}
And since the graph is bipartite we have:
\begin{align*}
    A_{2\ell}^{LR} &= \mathbf{0} \qquad,\qquad A_{2\ell}^{RL} = \mathbf{0} \\
    A_{2\ell+1}^{LL} &= \mathbf{0} \qquad,\qquad A_{2\ell+1}^{RR} = \mathbf{0} 
\end{align*}
These equations yield a recursive formula for $A_\ell^{LL}$, with the following initial conditions:
\begin{equation}\label{rec_eq:1}
\begin{split}
    A_2^{LL} &= M M^\top - cI \\
    A_4^{LL} &= M M^\top A_2^{LL} - (c-1+d-1)A_2^{LL} - c(d-1) I \\
    MM^\top A_\ell^{LL} &= A_{\ell+2}^{LL} + ((c-1)+(d-1))A_\ell^{LL} + (c-1)(d-1) A_{\ell-2}^{LL} \qquad,\qquad \forall \ell \geq 4
\end{split}
\end{equation}
The following lemma, proven in \nameref{app},
suggests a way to find a non-recursive formula for $A_{2\ell}^{LL}$,
given such linear recursive relations with fixed coefficients.

\begin{restatable}{lemma}{linreclemma}
\label{linrec}
Let $(x_n)$ be a series defined via a second order linear recurrence with fixed coefficients $A,B \in \mathbb{C}$:
\begin{equation*}
    x_n = A x_{n-1} + B x_{n-2}
\end{equation*}
Assume $\lambda_1 \neq \lambda_2$ are (real or complex) roots of the characteristic polynomial $\lambda^2 - A \lambda - B$.
Then there are $\alpha, \beta \in \mathbb{C}$, that depend on the initial conditions $x_0, x_1$, such that
\begin{equation*}
    \quad x_n = \alpha \lambda_1^n + \beta \lambda_2^n
\end{equation*}
for every $n \geq 0$.

If the characteristic polynomial has a single root $\lambda$ of multiplicity $2$, then there are $\alpha, \beta \in \mathbb{C}$ such that
\begin{equation*}
    x_n = \alpha \lambda^n + \beta n \lambda^n
\end{equation*}
for every $n \geq 0$.
\end{restatable}
We use the lemma to bound the eigenvalues of $A_{2\ell}^{LL}$ given bounds on the spectrum of the biregular graph.
\begin{lemma}\label{poly_bound}
Let $G$ be a $(c,d)$-biregular graph.
Then there is a sequence of polynomials with integer coefficients $(p_\ell(x))$ such that for every eigenpair $(\lambda, v)$ of $G$, $p_\ell(\lambda^2)$ is an eigenvalue of $A_{2\ell}^{LL}$, and moreover, for every $\lambda \in \mathbb{R}$, if
\begin{equation}\label{eq:lambda-bound}
    \abs{\lambda} \in \{0\} \cup [\sqrt{d-1} - \sqrt{c-1}, \sqrt{d-1}+\sqrt{c-1}]
\end{equation}
then
\begin{equation}\label{eq:pl-bound}
    \abs{p_\ell(\lambda^2)} \leq (2+\sqrt{d-1})\ell(c-1)^{\ell/2}(d-1)^{\ell/2}.
\end{equation}
\end{lemma}
\nnote{perhaps defer proof of lemma and first finish pf of thm 2}
\rnote{reordered. I can maybe finish proving \cref{upper_bound} first and then prove this lemma but I think it's OK now because the thm is already proven above.}
\begin{proof}
The recursive formulae proven above \eqref{rec_eq:1} suggest that there is a series of polynomials $p_n(x)$
with integer coefficients such that $A_{2n}^{LL} = p_n(M M^\top)$.
Note that the graph's adjacency matrix is
\begin{equation*}
    A_G = \begin{bmatrix}
        0 & M \\
        M^\top & 0
    \end{bmatrix}
\end{equation*}
And so, if $(\lambda, v)$ is an eigenpair of $G$, then $(\lambda^2, v)$ is an eigenpair of 
\begin{equation*}
A_G^2 = \begin{bmatrix}
    M M^\top & 0 \\ 0 & M^\top M
\end{bmatrix}
.
\end{equation*}
This shows that $p_\ell(\lambda^2)$ is an eigenvalue of $A_{2\ell}^{LL}$ whenever $\lambda$ is an eigenvalue of $G$.
The converse is also true.

The formulae \eqref{rec_eq:1} can be transformed so as to convey that
$p_n(x)$ satisfies these equations:
\begin{gather*}
    p_1(x) = x-c \qquad , \qquad
    p_2(x) = x^2 + (2-2c-d)x + c(c-1) \\
    xp_n(x) = p_{n+1}(x) + (c-1+d-1)p_n(x) + (c-1)(d-1) p_{n-1}(x) 
\end{gather*}
for all $n>1$.
Setting $n=1$ gives an equation involving $p_0(x), p_1(x), p_2(x)$.
We can solve this equation for $p_0(x)$
and get a simpler description of the initial conditions:
\begin{gather}
    p_0(x) = \frac{c}{c-1} \qquad , \qquad
    p_1(x) = x-c \label{req_eq:2} \\
    xp_n(x) = p_{n+1}(x) + (c-1+d-1)p_n(x) + (c-1)(d-1) p_{n-1}(x) 
    \label{req_eq:3}
\end{gather}
for all $n > 0$.

We fix some $t$
\nnote{why are we denoting it $t$ and not $\lambda$?}%
\rnote{because I reserved $\lambda$ for the roots of the characteristic polynomial associated with the linear recurssion. Maybe I should swap the two?}%
that satisfies \eqref{eq:lambda-bound}, namely such that \[\abs{t} \in \{0\} \cup [\sqrt{d-1} - \sqrt{c-1}, \sqrt{d-1}+\sqrt{c-1}].\]
We first deal with the case
where $\abs{t} \in (\sqrt{d-1}-\sqrt{c-1}, \sqrt{d-1}+\sqrt{c-1})$, and later we will consider the edge cases where $t$ is one of the endpoints of the segment or $0$.
Let us write $x = t^2$.
We have that for this fixed $x$, the series $(p_n(x))_n$ satisfies a second order linear recurrence with fixed coefficients.
Using \cref{linrec}, we conclude that there are functions $\alpha(x), \lambda_1(x), \beta(x), \lambda_2(x)$ that depend only on $x, c$ and $d$, such that
\begin{equation}\label{eq:pn}
    p_n(x) =
    \alpha(x) (\lambda_1(x))^n + 
    \beta(x) (\lambda_2(x))^n
\end{equation}
for every $n$.

In order to find $\lambda_1, \lambda_2$ we solve for $\lambda$ the characteristic polynomial,
namely the following quadratic equation derived from \eqref{req_eq:3}:
\begin{equation*}
    x\lambda = \lambda^2 + (c-1+d-1)\lambda + (c-1)(d-1)
\end{equation*}
To obtain
\begin{equation*}
    \lambda_{1,2}(x) = \frac{x - (c-1) - (d-1) \pm \sqrt{\Delta(x)}}{2}
\end{equation*}
where
\begin{equation}\label{eq:delta}
    \Delta(x) = x^2 - 2x((c-1)+(d-1)) + (c-d)^2
    .
\end{equation}
Using the initial values for $p_0(x), p_1(x)$ from \eqref{req_eq:2}, and plugging back into \eqref{eq:pn}
we get the equations
\begin{align*}
    \frac{c}{c-1} &= \alpha(x) (\lambda_1(x))^0 + \beta(x) (\lambda_2(x))^0 = \alpha(x) + \beta(x) \\
    x-c &= \alpha(x) (\lambda_1(x))^1 + \beta(x) (\lambda_2(x))^1 = \alpha(x) \lambda_1(x) + \beta(x) \lambda_2(x)
\end{align*}
whose solution is
\begin{align*}
    \alpha(x) &= \frac{(c-1)x - (c-1)^2 - (c-1) + (c-1)(d-1) + (c-1) \sqrt{\Delta(x)} - x + d-1 + \sqrt{\Delta(x)}}{2(c-1)\sqrt{\Delta(x)}} \\
    \beta(x) &= \frac{c}{c-1} - \alpha(x)
    .
\end{align*} 
We check when $\Delta(x) = 0$ by solving \eqref{eq:delta} for $x$:
\begin{align*}
    x_{1,2} &=
    \frac{2((c-1)+(d-1)) \pm \sqrt{4(c-1+d-1)^2 - 4(c-d)^2}}{2} \\&=
    (c-1+d-1) \pm \sqrt{(c+d)^2 - 4(c+d) + 4 - (c-d)^2} \\&=
    (c-1+d-1) \pm \sqrt{c^2 + 2cd + d^2 - 4c - 4d + 4 - c^2 + 2cd - d^2} \\&=
    (c-1+d-1) \pm \sqrt{4cd - 4c - 4d + 4} \\&=
    (c-1+d-1) \pm 2\sqrt{c-1}\sqrt{d-1} \\&=
    (\sqrt{d-1} \pm \sqrt{c-1})^2
\end{align*}
We see that $\Delta(x)$ is quadratic in $x$ and has roots at $(\sqrt{d-1} \pm \sqrt{c-1})^2$.
This gives a nice factorization of $\Delta(x)$:
\begin{align*}
    \Delta(x) &=
    x^2 - 2x((c-1)+(d-1)) + (c-d)^2 \\&=
    \pare{x - \pare{\sqrt{d-1} + \sqrt{c-1}}^2}
    \pare{x - \pare{\sqrt{d-1} - \sqrt{c-1}}^2}
\end{align*}
Recall that for the $x$ we fixed
\nnote{add ref to appropriate inequality here}%
\rnote{done}%
we have $\sqrt{x} = t \in (\sqrt{d-1}-\sqrt{c-1}, \sqrt{d-1}+\sqrt{c-1})$,
so the first term in the product is negative and the second term is positive,
so $\Delta < 0$, and so $\lambda_{1,2}$ are complex numbers (conjugate to one another),
with magnitude
\begin{equation}
\begin{split} \label{eq:magic-cancellation}
        \abs{\lambda_{1,2}}^2 &=
    \frac{(x-(c-1)-(d-1))^2 - \Delta(x)}{4} \\&=
    \frac{x^2 - 2x((c-1)+(d-1)) + (c-1+d-1)^2 - (x^2 - 2x((c-1)+(d-1)) + (c-d)^2)}{4} 
    \\&=
    \frac{(c+d-2)^2 - (c-d)^2}{4} = (c-1)(d-1)
\end{split}
\end{equation}
A very similar calculation shows that $\alpha, \beta$ are conjugates with magnitude
\begin{equation*}
    \abs{\alpha}^2 = \abs{\beta}^2 = \frac{x(x-cd)}{\Delta(x) \cdot (c-1)}
\end{equation*}
This finishes the proof for all such $x$'s:
\begin{align*}
    \abs{p_\ell(x)} &=
    \abs{\alpha(x) \lambda_1^\ell + \beta(x) \lambda_2^\ell} \leq
    \abs{\alpha(x) \lambda_1^\ell} + \abs{\beta(x) \lambda_2^\ell} \\&=
    \abs{\alpha(x)} \abs{\lambda_1}^\ell + \abs{\beta(x)} \abs{\lambda_2}^\ell \\&=
    2\sqrt{\frac{x(x-cd)}{\Delta(x) \cdot (c-1)}} (c-1)^{\ell/2} (d-1)^{\ell/2}
\end{align*}
We keep in mind that $x$ is fixed, so the expression is smaller than
$(2+\sqrt{d-1})\cdot \ell \cdot(c-1)^{\ell/2} (d-1)^{\ell/2}$ for large enough $\ell$.

We are left with the cases $x = t^2$ for $t = 0, \sqrt{d-1} \pm \sqrt{c-1}$:
\begin{enumerate}
    \item $t = 0$.
    We use the same methods and find that the characteristic polynomial is
    \begin{equation*}
        \lambda^2 + (c-1+d-1)\lambda + (c-1)(d-1)
    \end{equation*}
    whose roots are
    \begin{equation*}
        \lambda_1 = -(c-1) \qquad,\qquad
        \lambda_2 = -(d-1)
        .
    \end{equation*}
    Using the initial conditions ($p_0(0) = c/(c-1)$, $p_1(0) = -c$) we obtain
    \begin{equation*}
        \alpha(0) = \frac{c}{c-1} \qquad,\qquad
        \beta(0) = 0
    \end{equation*}
    and using the fact that $c < d$ we get
    \begin{align*}
        \abs{p_\ell(0)} &=
        \abs{\alpha(0)\lambda_1^\ell + \beta(0)\lambda_2^\ell} \\&=
        \frac{c}{c-1} (c-1)^\ell \\&<
        2 l (c-1)^{\ell/2} (c-1)^{\ell/2} \\&<
        2 l (c-1)^{\ell/2} (d-1)^{\ell/2}
        .
    \end{align*}
    \item $t = \sqrt{d-1} + \sqrt{c-1}$.
    Then $x = t^2 = (\sqrt{d-1} + \sqrt{c-1})^2 = d-1 + c-1 + 2\sqrt{d-1}\sqrt{c-1}$,
    and the characteristic polynomial has a single root of multiplicity $2$, namely
    \begin{equation*}
        \lambda = \frac{x - (c-1) - (d-1)}{2} = \sqrt{d-1}\sqrt{c-1}
        .
    \end{equation*}
    The solution, therefore, takes the form
    \begin{equation*}
        p_n(x) = (\alpha(x) + n \beta(x)) (c-1)^{n/2} (d-1)^{n/2}
        .
    \end{equation*}
    Using the initial values we get
    \begin{equation*}
        \alpha(x) = \frac{c}{c-1} \qquad,\qquad
        \beta(x) = \frac{x-c}{\sqrt{d-1}\sqrt{c-1}} - \frac{c}{c-1} =
        2 + \frac{d-2}{\sqrt{d-1}\sqrt{c-1}} - \frac{c}{c-1}
        .
    \end{equation*}
    $1 < \frac{c}{c-1} \leq 2$ so $\beta(x) \leq \sqrt{d-1} + 1$,
    and in total we get
    \begin{align*}
        \abs{p_\ell(x)} &=
        \abs{\alpha(x) + \ell\beta(x)} (c-1)^{\ell/2} (d-1)^{\ell/2} \\&\leq
        \pare{\left| \frac{1}{\ell} \cdot \frac{c}{c-1} \right| + \abs{\beta(x)}} \ell (c-1)^{\ell/2} (d-1)^{\ell/2} \\&\leq
        \pare{2+\sqrt{d-1}} \ell (c-1)^{\ell/2} (d-1)^{\ell/2}
    \end{align*}
    For sufficiently large $\ell$.
    \item $t = \sqrt{d-1} - \sqrt{c-1}$.
    We get $x = t^2 = d-1 + c-1 - 2\sqrt{d-1}\sqrt{c-1}$, and the rest follows the same calculations as in the previous case.
\end{enumerate}
\end{proof}

Bounds on the spectrum of $A_{2\ell}^{LL}$ give bounds on the number of non-backtracking paths completely contained in a small set, hence gives \cref{upper_bound}.

\begin{proof}[Proof of \cref{upper_bound}]
Recall that $M_{2\ell}(S,G)$ counts the number of non-backtracking paths of length $2\ell$ that start and end in $S$,
so by the definition of the $A_n^{LL}$ operatore, we have
$M_{2\ell}(S,G) = \langle A_{2\ell}^{LL} \mathbbm{1}_S, \mathbbm{1}_S \rangle$.

We note that $A_{2\ell}^{LL} \mathbbm{1}_L = c(c-1)^{\ell-1} (d-1)^\ell \mathbbm{1}_L$,
because every non-backtracking path
starting at a given vertex
is made of picking the first left-to-right edge
(we have $c$ such edges to pick from),
and then alternating between picking any of the $d$ or $c$ edges adjacent to the current vertex,
except for the edge we picked to get to it.

Write $\mathbbm{1}_S = \frac{\abs{S}}{n} \mathbbm{1}_L + r$, with $r \perp \mathbbm{1}_L$,
and $\norm{r}_2^2 \leq \norm{\mathbbm{1}_S}_2^2 = \abs{S}$.
Since the graph is Ramanujan,
the nontrivial eigenvalues in its spectral decomposition have their absolute value in the set
$\{0\} \cup [\sqrt{d-1}-\sqrt{c-1}, \sqrt{d-1}+\sqrt{c-1}]$.
We only care about the nontrivial eigenvalues because $r \perp \mathbbm{1}_L$,
hence in the writing of $r$ in the orthogonal basis made of eigenvectors,
only eigenvectors with nontrivial eigenvalues appear.
We use \cref{poly_bound} to get
\begin{equation*}
    \langle A_{2\ell}^{LL} r, r \rangle \leq
    (2+\sqrt{d-1})\ell (c-1)^{\ell/2} (d-1)^{\ell/2} \cdot \norm{r}_2^2.
\end{equation*}
Combine everything to get
\begin{align*}
    M_{2\ell}(S,G) &=
    \langle A_{2\ell}^{LL} \mathbbm{1}_S, \mathbbm{1}_S \rangle =
    \left\langle A_{2\ell}^{LL} \frac{\abs{S}}{n} \mathbbm{1}_L + r, \frac{\abs{S}}{n} \mathbbm{1}_L + r \right\rangle \\&=
    \frac{\abs{S}^2}{n^2} \langle A_{2\ell}^{LL} \mathbbm{1}_L, \mathbbm{1}_L \rangle +
    \langle A_{2\ell}^{LL} r, r \rangle \\&=
    \frac{\abs{S}^2}{n^2} \cdot c(c-1)^{\ell-1}(d-1)^\ell \langle \mathbbm{1}_L, \mathbbm{1}_L \rangle +
    \langle A_{2\ell}^{LL} r, r \rangle \\&\leq
    \frac{\abs{S}^2}{n} c(c-1)^{\ell-1} (d-1)^\ell + 
    (2+\sqrt{d-1}) \ell(c-1)^{\ell/2}(d-1)^{\ell/2} \norm{r}_2^2 \\&\leq
    \abs{S} \left(\frac{\abs{S} \cdot c \cdot (c-1)^{\ell/2} (d-1)^{\ell/2}}{n (c-1)} + (2+\sqrt{d-1})\ell \right) (c-1)^{\ell/2} (d-1)^{\ell/2} \\&\leq
    \abs{S} \left(\frac{c}{c-1} + (2+\sqrt{d-1})\ell\right) (c-1)^{\ell/2} (d-1)^{\ell/2} \\&\leq
    \abs{S} \pare{(2+\sqrt{d-1})\ell+2} (c-1)^{\ell/2} (d-1)^{\ell/2}
\end{align*}
\end{proof}

\section{Random gadget}\label{gadget_section}
In this section we prove the existence of bipartite graphs such that every small set of left-side vertices
has a unique neighbour on the right side.
We draw a random biregular graph from a similar distribution as in \cite{pippenger1977superconcentrators},
and use techniques similiar to \cite[Thm 4.4]{vadhan2012pseudorandomness}.

\nnote{more convenient for the reader if this moves to the beginning of the section}
\rnote{done}
\begin{lemma}\label{gadget}
For every integers $L, R, c, d$ with $Lc = Rd$, $L>R$, $c>3$, if $k$ is an integer that satisfies the inequality
\begin{equation*}\label{gadget_k_bound}
    k^{\frac{c-3}{2}} \leq 
    \frac{1}{2Le} \cdot \pare{\frac{R}{3ec}}^{\frac{c-1}{2}}
\end{equation*}
then there is a $(c,d)$-biregular graph with sides $[L]$ and $[R]$
such that every set of left vertices of size at most $k$
has a unique neighbour.
\end{lemma}

We draw a random $(c,d)$-biregular graph in the following way:
fix $L$ vertices on the left side and $R$ vertices on the right side ($cL = dR$),
write $c$ copies of each left-side vertex and $d$ copies of each right-side vertex, and connect them via a uniformly random matching.
That is, pick a uniformly random permutation $\pi: L \times [c] \to R \times [d]$,
and for every $u \in L, v \in R, i \in [c], j \in [d]$,
if $\pi(u,i) = (v,j)$,
then add $(u,v)$ as an edge.
Note that we allow multiple edges between two vertices (if there are several $i,j$ satisfying $\pi(u,i) = (v,j)$).

Let $G$ be a random bipartite graph with $L$ vertices on the left side and $R$ vertices on the right side drawn from said distribution.
Let $A$ be a subset of left vertices of size $k$.
We note that if $A$ expands by at least $(c+1)/2$,
then, by a simple counting argument, $A$ has a unique neighbour.
It is therefore sufficient to find the probability that $A$ expands by at least $(c+1)/2$.

Let us fix an arbitary ordering of the $ck$ edges leaving $A$, and denote it $e_1, \ldots, e_{ck}$.
We say that $e_i$ is a \textit{repeat} if it touches a previously covered vertex,
that is,
if its right endpoint is contained in the set of right endpoints of the set
$e_1, \ldots, e_{i-1}$.
We note that if $A$ does \emph{not} expand by at least $(c+1)/2$, then,
again by a simple counting argument,
there are at least $(c-1)k/2$ repeats.
This is because the number of repeats and the size of the set of the neighbours of $A$
add up to the number of edges leaving $A$, namely $ck$.

We note that for every $i$, $e_i$ is a repeat if it touches one of $i-1$ or less previously covered vertices.
This means that $\prob{\text{$e_i$ is a repeat}} \leq \frac{i-1}{R} < \frac{ck}{R}$.
Moreover, if we condition on the event that some of the first $i-1$ edges are also repeats, then the probability that $e_i$ is a repeat may only decrease, since it means that there are less \say{forbidden} endpoints.
We conclude that for every set of $l$ edges:

\begin{align*}
    \prob{\text{$e_{i_1}, \ldots, e_{i_l}$ are repeats}} &=
    \prod_{j=1}^l \prob{\text{$e_{i_j}$ is a repeat} \given \text{$e_{i_1}, \ldots e_{i_{j-1}}$ are repeats}} <
    \pare{\frac{ck}{R}}^l
    .
\end{align*}

If $A$ expands too little, then there are many repeats.
We can use it to bound the probability that $A$ has no unqiue neighbour:
\begin{align*}
    \prob{\text{$A$ has no unique neighbour}} &\leq
    \prob{\text{$A$ expands by $<(c+1)/2$}} \\&\leq
    \prob{\text{there are at least $(c-1)k/2$ \textit{repeats}}} \\&\leq
    \sum_{{i_1}, \ldots, {i_{(c-1)k/2}} \in \binom{ck}{(c-1)k/2}}
    \prob{\text{$\{e_{i_1}, \ldots, e_{i_{(c-1)k/2}}\}$ are repeats}}
    \\&<
    \binom{ck}{\frac{c-1}{2}k} \cdot
    \pare{\frac{ck}{R}}^{\frac{c-1}{2}k}
\end{align*}

And by a union bound over the possible choices of $A$:
\begin{align}\label{prob_bad_set}
\begin{split}
    \prob{\exists\: \text{\say{bad} $A$ of size $k$}} &\leq
    \binom{L}{k} \cdot \prob{\text{$A$ expands by $<(c+1)/2$}} \\&\leq
    \binom{L}{k} \cdot
    \binom{ck}{\frac{c-1}{2}k} \cdot
    \pare{\frac{ck}{R}}^{\frac{c-1}{2}k} \\&\leq
    \pare{\frac{L e}{k}}^k \cdot
    \pare{\frac{cke}{\frac{c-1}{2}k}}^{\frac{c-1}{2}k} \cdot
    \pare{\frac{ck}{R}}^{\frac{c-1}{2}k} \\&=
    \pare{ \frac{Le}{k} \cdot \pare{\frac{2ce}{c-1} \cdot \frac{ck}{R}}^{\frac{c-1}{2}} }^k \\&\leq
    \pare{ \frac{Le}{k} \cdot \pare{\frac{3eck}{R}}^{\frac{c-1}{2}} }^k 
\end{split}
\end{align}

Where the last inequality follows from assuming that $c \geq 3$ so $\frac{2c}{c-1} \leq 3$.

We are now ready to prove \cref{gadget}.
\begin{proof}[Proof of \cref{gadget}]
Let us draw a $(c,d)$-biregular graph $G = ([L] \sqcup [R], E)$ from the distribution described above.
Let $k$ be an integer satsifying \eqref{gadget_k_bound}.
Using a union bound and the inequality in \eqref{prob_bad_set}, we have
(where probability is taken over the choice of $G$):
\begin{align*}
    \prob{\exists\: \text{\say{bad} $A \subseteq [L]$ of size $\leq k$}} &=
    \sum_{a=1}^k \prob{\exists\: \text{\say{bad} $A \subseteq [L]$ of size $a$}} \\&\leq
    \sum_{a=1}^k   \pare{ \frac{Le}{a} \cdot \pare{\frac{3eca}{R}}^{\frac{c-1}{2}} }^a  \\&<
    \sum_{a=1}^\infty  \pare{ \frac{Le}{k} \cdot \pare{\frac{3eck}{R}}^{\frac{c-1}{2}} }^a \\&=
    \sum_{a=1}^\infty  \pare{ k^{\frac{c-1}{3}} \cdot Le \cdot \pare{\frac{3ec}{R}}^{\frac{c-1}{2}} }^a 
    \\&\leq
    \sum_{a=1}^\infty  \pare{\frac{1}{2}}^a <
    1
    .
\end{align*}
We see that with strictly positive probability,
a random graph has no \say{bad} subsets of size $\leq k$,
hence there exists a graph with the desired unique neighbour property.
\end{proof}
\nnote{Excellent}
\section{Construction}\label{construction}
\subsection{Routed product definition}\label{rpd}
Let us begin with a brief coding theory motivation. An error-correcting code is often given via an $m\times n$ parity check matrix $H$, so that $C = \operatorname{Ker} H\subseteq \{0,1\}^n$. The matrix $H$ can be visualized as a bipartite graph, called the {\em parity check graph}, with $n$ left and $m$ right vertices, and an edge $i\sim j$ whenever $H(j,i)\neq 0$. 
A Tanner code is defined given a bipartite graph $B$ and a base code $C_0 = \operatorname{Ker} H_0$ \cite{tanner1981recursive}. One way to view the routed product is through the point of view of codes. Consider the parity check graph $B_0$ of $H_0$ and define the \emph{routed product} of $B$ and $B_0$ to be simply the parity check graph of the Tanner code ${\cal C}(B,C_0)$. 

Here is a more detailed and combinatorial definition of the routed product without mention of codes. 
Let $G = (L \sqcup R, E)$ be a $(c,d)$-biregular graph
and $G_0 = (L_0 \sqcup R_0, E_0)$ a $(c_0, d_0)$-biregular graph.
\nnote{add: The composed graph is modeled after Tanner codes \cite{tanner1981recursive}. The constraint graph of the Tanner code is the result of a ``routed product'' of a big graph $G$ with a graph $G_0$ that is the constraint graph of a base code $C_0$. }
\rnote{added at the top of this subsection. It's also in the related work section}
We think of $G$ as a big graph (in practice, an infinite family of Ramanujan graphs),
and $G_0$ as a fixed size graph (gadget).
Assume that $\abs{L_0} = d$,
and let us think of the edges of $G$ as a function $E: R \times [d] \to L$ which maps a right side vertex $v$ and an index $i$ to the $i^\text{th}$ neighbour of $v$ in $G$.

We can define the \emph{routed product} graph $G' = G \routeprod G_0$
as the bipartite graph whose left side is $L$, right side is the cartesian product $R \times R_0$, and the set of edges is
\begin{equation*}
    E' = \{\pare{E(v,i), (v, j)}:\: v \in R, i \in [d], j \in [R_0], (i,j) \in E_0\}
    .
\end{equation*}
That is, we write $R_0$ copies of each vertex in $R$,
and every right side vertex $v$ in the big graph $G$
and an edge $(i,j)$ in the small gadget $G_0$
gives an edge between the $i^\text{th}$ neighbour of $v$ in $G$,
and the $j^\text{th}$ vertex of the copy of $G_0$ assigned to $v$ in $G'$.
Otherwise put, we use $G_0$ to route every edge of the big graph $G$
to $c_0$ edges in the product graph $G'$.

More precisely,
for every $v \in R$,
the bipartite subgraph of $G'$ whose left side is $N_{G}(v)$ and right side is $(v, \cdot)$ is isomorphic to $G_0$.
This means that, roughly speaking, unique neighbours are inherited from the small graph to the product graph:
\begin{restatable}{lemma}{unqclaim}\label{unq_claim}
Let $S \subseteq L$, $v \in N_{G}(S)$.
Define $S' = \{i: E(v, i) \in S\} \subseteq [d]$ as the indexed neighbours of $v$ in $S$.
If $S'$, as a set of vertices in the gadget $G_0$, has a unique neighbour $j \in R_0$ in $G_0$,
then $(v,j)$ is a unique neighbour of $S$ in the product graph $G'$.
\end{restatable}
The proof is immediate while staring at \cref{fig:my_label},
but for the sake of completion it is given in \nameref{app}. \nnote{the appendix; or Appendix A}
\rnote{done}

\begin{figure}
    \centering
    \includegraphics[height=0.65\paperheight,keepaspectratio]{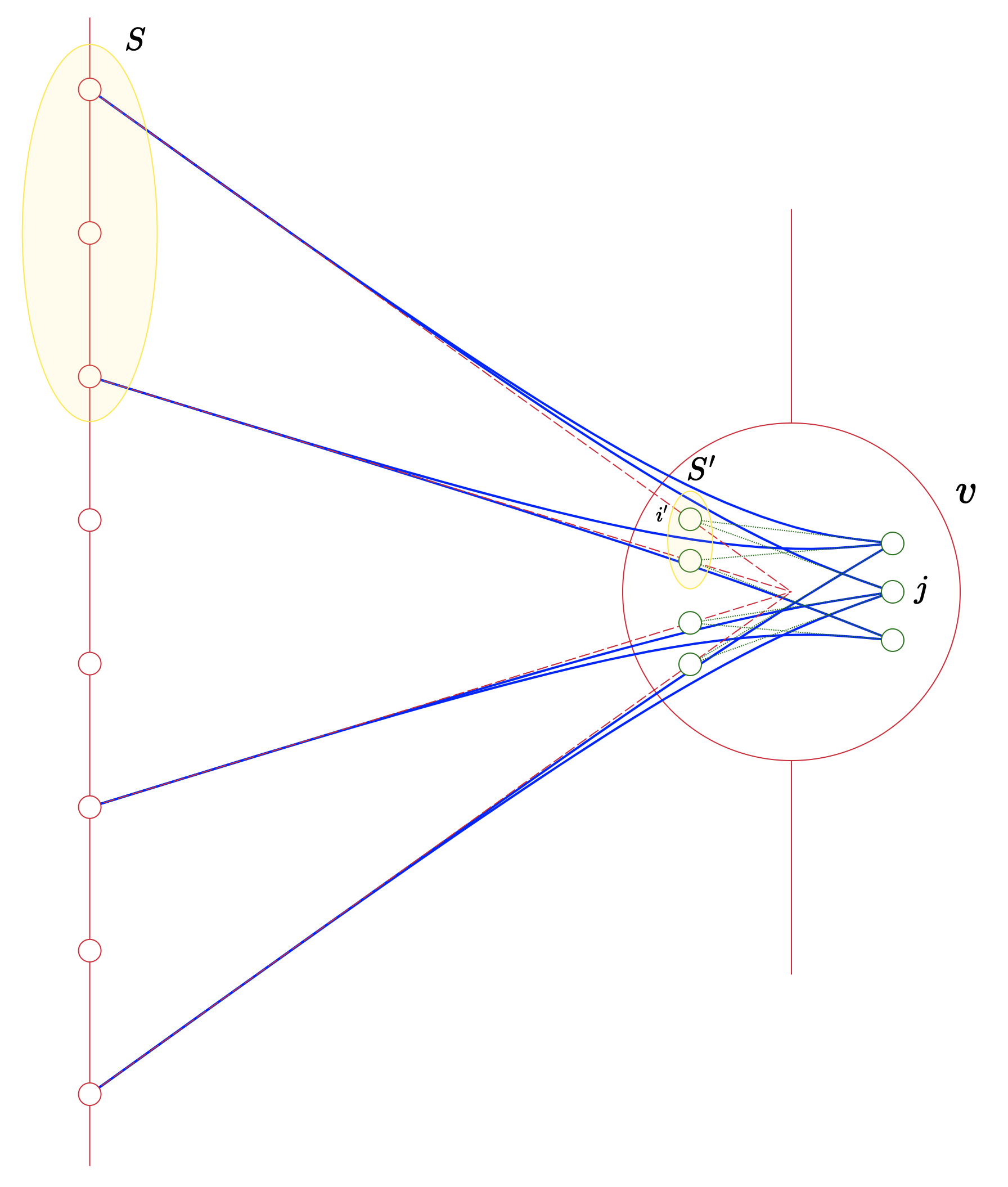}
    \caption{An example of a bipartite graph $G$ (dashed, red),
    a small gadget $G_0$ (dotted, green),
    and the routed product $G' = G \routeprod G_0$ (solid, blue).
    The set $S \subseteq L$ has a neighbour $v \in R$,
    and so $S$ is associated with a set $S'$ of left side vertices of the copy of $G_0$ associated with $v$.
    Since $(i', j)$ is the only edge connecting $j$ to $S'$ in $G_0$,
    we have that $(v,j)$ is a unique neighbour of $S$ in $G'$.
    }
    \label{fig:my_label}
\end{figure}

\subsection{Proof of Theorem \ref{une}}\label{thmproof}
Let $q$ be a prime power, $c_0$ an integer, and $\alpha > 1$.
Assume that $\alpha c_0 (q+1)$ is an integer.
We construct an infinite family of $(c_0(q+1), \alpha c_0 (q+1))$-biregular
graphs with the unique neighbour property
under some assumptions specified below.

Denote $c = q+1$ and $d = q^3+1$.
By \cref{ballentine} \nnote{refer to precise statement of theorem from prelims}\rnote{done}
there is an efficient construction of an infinite family of $(c,d)$-biregular Ramanujan graphs $(G_n)$.
Let $G_0 = (L_0 \sqcup R_0, E_0)$ be a gadget: a $c_0$-left-regular bipartite graph with $\abs{L_0} = d = q^3+1$ vertices on the left side and $R_0$ vertices on the right side, such that every left-side set of sufficiently small size admits a unique neighbour on the right side, where \say{sufficiently small} here means the bound given in \cref{gadget}.
For the constructed graph to have the left side $\alpha$ times bigger than the right side,
we set $R_0 = \frac{d}{\alpha c} = \frac{q^3+1}{\alpha(q+1)}$.

We define $G'_n = G_n \routeprod G_0$ as the routed product of $G_n$ and $G_0$. For the rest of this (short) proof let us suppress $n$ from the notation, for convenience.

Let $\eps < \frac 1 q$.
By \cref{thm:main-tech}, there exists $\delta > 0$
such that for every $S \subseteq L$ of size at most $\delta \abs{L}$,
the \say{average right degree} $\bar d_S$,
namely the average of the degrees of vertices in $N_G(S)$ in the induced subgraph $S \sqcup N_G(S)$,
is bounded:
\begin{equation*}
    \bar d_S :=
    \frac{c \abs{S}}{\abs{N_G(S)}}
    \leq
    1 + (1+\eps)\sqrt{\frac{d-1}{c-1}}
    .
\end{equation*}
We show that such $S$ has a unique neighbour in $G'$.

We note that $\frac{d-1}{c-1} = q^2$,
so since $\eps<\frac 1 q$ we have
a vertex $v \in R$ of \say{degree} at most $q+1$ in $G$,
that is, 
the set $S' \subseteq [d]$ of $v$'s neighbours in $S$ is of size at most $q+1$.
By \cref{unq_claim}, if $S'$, as a set of left-side vertices in $G_0$,
has a unique neighbour $j$ \nnote{check} in $G_0$,
then our original set $S$ has a unique neighbour $(v,j)$  \nnote{check} in $G$.
\rnote{correct}

It remains to choose the parameters in a way that all left-side sets of size at most $q+1$
have a unique neighbour in $G_0$.
By \cref{gadget}, we need to have:
\begin{equation}\label{ineq_q}
    (q+1)^{\frac{c_0-3}{2}} \leq
    \frac{1}{2(q^3+1) e} \cdot
    \pare{\frac{\frac{q^3+1}{\alpha(q+1)}}{3 e c_0}}^{\frac{c_0-1}{2}}
    .
\end{equation}
The LHS is $O(q^{\frac{c_0-3}{2}})$ and RHS is $\Theta(q^{c_0-4})$,
so if $c_0 > 5$ then for sufficiently large $q$ the construction gives a unique neighbour expander.
That is, there exists some $\hat q(c_0, \alpha)$ such that if $q > \hat q$
then \eqref{ineq_q} holds, hence we constructed a bipartite unique neighbour expander
as promised in \cref{une}.

\section{Future work}
\nnote{distinguish between theoretically efficient, which it is; and between "practical" which it isn't. Also add a couple of words about random but verifiable}
\rnote{done}
The main pitfall of our approach is the non-constructive nature of the gadget. Theoretically since the gadget has constant size this is no issue. However, exhaustive search is impractical even for small values of $q$. 
This is because the gadget's size is cubic in $q$ so the search space is of size exponential in $q^3$.
A natural question would be whether it is possible to construct such a gadget in an efficient way,
since that would lead to the whole unique neighbour expander family to be constructible in practice.
For the bipartite Ramanujan family chosen in our work (the one by Ballantine et al. \cite{ballantine2015explicit})
we ask the following.
\begin{question}
For which prime power $q$ and real number $\alpha \geq 1$
can one construct efficiently a biregular graph
with left side $q^3+1$,
right side $\frac{q^3+1}{\alpha(q+1)}$,
such that every left side set of size at most $q+1$ has a unique neighbour?
\end{question}

We note that the fixed size graph given in \cite[Lemma 4.3]{alon2002explicit}
is a good gadget 
(for $\alpha = 22/21$ and the edge-vertex incidence graphs of a $44$-regular Ramanujan graph family),
and indeed these graphs can be used to construct bipartite unique neighbour expanders.

Since we prove that a random gadget is, with non-negligble probability,
good for our construction, it may be interesting to construct
such gadget by simply drawing random gadgets and testing whether they are good.
Since drawing is simple, we are left with the task of testing. We therefore ask:
\begin{question}
    Given a bipartite graph,
    can one efficiently find the smallest nonempty
    set of left-side vertices that has no unique neighbours?
\end{question}
We currently know of no better way than just enumerating all left-side sets,
which is exponential in the size of the graph, hence impractical.
We refer to \cite{applebaum2019sampling} for an interesting approach to testing expansion of random graphs.

The methods presented in this work are not limited to the $(q+1,q^3+1)$-biregular Ramanujan family.
We can therefore ask the question the other way around -- find a gadget (by sampling or any other way),
and see whether we can efficiently construct a bipartite Ramanujan family that will
make it work, i.e. that would allow us to rewrite the proof of \cref{une}.
This emphasizes the well-known natural question of constructing Ramanujan graphs with arbitrary degrees, specifically in the bipartite and biregular setting,
\begin{question}
For which integers $c < d$ can one construct efficiently
an infinite family of $(c,d)$-biregular Ramanujan graphs?
\end{question}
We note that our construction is far from \say{right-side unique neighbour expansion,}
as the complete right side of a single gadget
is a constant-size set with no unique neighbours on the left.
We wonder whether it is possible to construct a bipartite graph where \emph{all}
small size sets (be them contained in either sides, or both) have unique neighbours.

\newpage
\printbibliography
\newpage
\section{Appendix A}\label{app}
We restate and prove the lemmas we used throughout the work.
\vve*
\begin{proof}
    Let $G = (V,E)$ a $d$-regular Ramanujan graph.
    The adjacency matrix of $G'$ is $A = \begin{bmatrix} 0 & M \\ M^\top & 0 \end{bmatrix}$ where $M$
    has $\abs{E}$ rows, each containing two 1's,
    and $\abs{V}$ columns, each containing $d$ 1's.
    Let $v$ be an eigenvector of $A$ with eigenvalue $\lambda$; then $v$
    is an eigenvector of $A^2$ with eigenvalue $\lambda^2$.
    We note that
    \begin{equation*}
        A^2 = \begin{bmatrix} M M^\top & 0 \\ 0 & M^\top M \end{bmatrix}
    \end{equation*}
    so it suffices to consider the spectrum of $M^\top M$, which is essentially the operator corresponding to a walk from a vertex of $G$ to an edge that touches it and back to one of its endpoints (possibly the same vertex we started at).
    
    For every $v \in V$, there are $d$ ways to walk from it to an edge and then back to $v$;
    all other legal paths correspond to picking an edge touching $v$.
    We conclude that $M^\top M = dI + A$, so every eigenvalue $\lambda$ of $G'$ satisfies
    $\lambda^2 = d + \sigma$ where $\sigma$ is an eigenvalue of $G$.

    The lemma is proven by noting that $\abs{\sigma} \leq 2\sqrt{d-1}$ (since $G$ is Ramanujan),
    so
    \begin{equation*}
        d-2\sqrt{d-1} \leq \lambda^2 \leq d+2\sqrt{d-1}
    \end{equation*}
    The terms on the extreme sides of the inequality can be verified to be
    $(\sqrt{d-1} \pm 1)^2$
    so we get $\abs{\lambda} \in [\sqrt{d-1}-1, \sqrt{d-1}+1]$, as needed (recall that in $G'$ the left-regularity is $c=2$ so $\sqrt{c-1}=1$).
\end{proof}
\linreclemma*
\begin{proof}
We note that for every $n \geq 2$ we have
\begin{equation*}
    \begin{bmatrix}
    x_n \\ x_{n-1}
    \end{bmatrix}
    =
    \begin{bmatrix}
    Ax_{n-1} + Bx_{n-2} \\ x_{n-1}
    \end{bmatrix}
    =
    \begin{bmatrix}
    A & B \\
    1 & 0
    \end{bmatrix}
    \begin{bmatrix}
    x_{n-1} \\ x_{n-2}
    \end{bmatrix}
\end{equation*}
Denote the $2 \times 2$ matrix by $D$, 
so by induction,
\begin{equation*}
    \begin{bmatrix}
    x_n \\ x_{n-1}
    \end{bmatrix}
    =
    D ^ n
    \begin{bmatrix}
    x_1 \\ x_0
    \end{bmatrix}
\end{equation*}
Let us diagonalize $D$.
The characteristic polyonmial is
\begin{align*}
p_D(\lambda) =
\operatorname{det} (\lambda I - D) =
\begin{vmatrix}
\lambda - A & -B \\
-1 & \lambda
\end{vmatrix} =
\lambda(\lambda-A) - B =
\lambda^2 -A \lambda - B
\end{align*}
If $p_D(\lambda)$ has two distinct roots $\lambda_1, \lambda_2$,
then the matrix is diagonalizable;
that means that there exists a $2 \times 2$ matrix $M$ such that $D = M \cdot \operatorname{diag} \{\lambda_1, \lambda_2\} \cdot M^{-1}$.
We get:
\begin{equation*}
    \begin{bmatrix}
    x_n \\ x_{n-1}
    \end{bmatrix}
    =
    M
    \begin{bmatrix}
    \lambda_1 & 0 \\
    0 & \lambda_2
    \end{bmatrix} ^ n
    M^{-1}
    \begin{bmatrix}
    x_1 \\ x_0
    \end{bmatrix}
    =
    M
    \begin{bmatrix}
    \lambda_1^n & 0 \\
    0 & \lambda_2^n
    \end{bmatrix}
    M^{-1}
    \begin{bmatrix}
    x_1 \\ x_0
    \end{bmatrix}
\end{equation*}
We can compute $M, M^{-1}$ explicitly, multiple the matrices and get $\alpha, \beta \in \mathbb{C}$ such that
$x_n = \alpha \lambda_1^n + \beta \lambda_2^n$ as required.

Otherwise, if $p_D(\lambda)$ has a single root $\lambda$ of multiplicity $2$,
then we can find its Jordan form, i.e. find $M$ such that
\begin{align*}
    D &= M
    \begin{bmatrix}
    \lambda & 1 \\
    0 & \lambda
    \end{bmatrix}
    M^{-1} \\
    D^n &= M
    \begin{bmatrix}
    \lambda & 1 \\
    0 & \lambda
    \end{bmatrix}
    ^n
    M^{-1}
    =
    M
    \begin{bmatrix}
    \lambda^n & n \lambda^{n-1} \\
    0 & \lambda^n
    \end{bmatrix}
    M^{-1}
\end{align*}
Where the last equality follows from a simple induction.

Similarly, we get
\begin{equation*}
    \begin{bmatrix}
    x_n \\ x_{n-1}
    \end{bmatrix}
    =
    M
    \begin{bmatrix}
    \lambda & 1 \\
    0 & \lambda
    \end{bmatrix} ^ n
    M^{-1}
    \begin{bmatrix}
    x_1 \\ x_0
    \end{bmatrix}
    =
    M
    \begin{bmatrix}
    \lambda^n & n \lambda^{n-1} \\
    0 & \lambda^n
    \end{bmatrix}
    M^{-1}
    \begin{bmatrix}
    x_1 \\ x_0
    \end{bmatrix}
\end{equation*}
And again we can find $\alpha, \beta \in \mathbb{C}$ as required.
\end{proof}

For the following lemma we remind that $G = (L \sqcup R, E)$ is a $(c,d)$-biregular graph,
$G_0 = (L_0 \sqcup R_0, E_0)$ is a $(c_0, d_0)$-biregular graph,
and $G' = G \routeprod G_0$ is the routed product of $G$ and $G_0$.
Recall that the edges of $G'$ are $(E(v,i), (v,j))$
when $v \in R$ is a right side vertex of $G$, $i \in [d]$, $E(v,i)$ is the $i^\text{th}$ neighbour
of $v$ according to $G$, and $(i,j) \in E_0$.
\unqclaim*
\begin{proof}
Assume that $i' \in S'$ is the unique neighbour of $j$ in $G_0$.
By the definition of the routed product we have that $(E(v,i'), (v,j))$ is an edge in $G$.
Since $i' \in S'$ we have that $E(v,i') \in S$,
so indeed $(v,j)$ is a neighbour of $S$ in $G'$.
It is therefore remaining to show that it is unique, i.e. that
$E(v,i')$ is the only neighbour of $(v,j)$ in $S$.

The neighbours of $(v,j)$ in $G$ are $E(v,i)$ for every $i$ such that $(i,j) \in E_0$.
If $E(v,i) \in S$, then by the definition of $S'$ we have that $i \in S'$,
so $i$ is a neighbour of $j$ in $E_0$.
But we know that $j$ is a unique neighbour of $S'$ in $E_0$,
so we must have that $i = i'$,
and indeed $(v,j)$ is a unique neighbour of $S$ in $G'$.
\end{proof}

\end{document}